\documentclass{article}
\author{Le Mai Nguyen}
\date{}

\usepackage{amsmath,amsthm,amssymb}
\theoremstyle{plain}
\newtheorem{thm}{Theorem}[section]
\newtheorem{prop}[thm]{Proposition}

\newtheorem{cons}[thm]{Construction}
\newtheorem{lem}[thm]{Lemma}

\theoremstyle{remark}
\newtheorem{rk}{\bf Remark}

\begin{document}
%\maketitle
\thispagestyle{empty}

\centerline {A STRICTLY STATIONARY, $M$-TUPLEWISE INDEPENDENT,}
\centerline {MIXING COUNTEREXAMPLE TO THE CENTRAL LIMIT THEOREM}  
\vskip 1 true in

\noindent Le Mai Nguyen\hfil\break

\noindent Department of Mathematics \hfil\break
Indiana University \hfil\break
Bloomington \hfil\break
Indiana 47405 \hfil\break
\vskip 2 true in

\noindent \copyright\ 2011 Le Mai Nguyen \hfil\break
\noindent All rights reserved. \hfil\break
\newpage
\section{Introduction}
\numberwithin{equation}{section}
\qquad Suppose $X:=(X_k,k\in\mathbf{Z})$ is a sequence of random variables on a probability space $(\Omega, {\cal F}, P)$.  For a given integer $M \geq 2$ the sequence is said to satisfy \textit{$M$-tuplewise independence} if for every choice of $M$ distinct integers, $k(1), k(2),..., k(M)$, the random variables $X_{k(1)}, X_{k(2)},...,X_{k(M)}$ are independent.  The sequence $X$ is said to be \textit{strictly stationary} if for all nonnegative integers $n$ and for all integers $i$ and $j$, the random vectors $(X_i, X_{i+1},..., X_{i+n})$ and $(X_j, X_{j+1},..., X_{j+n})$ have the same distribution.  If the sequence $X$ is strictly stationary, then it is said to be \textit{mixing (in the ergodic-theoretic sense)} if

\begin{equation}
\lim_{n \to \infty} P(X \in A \cap T^nB) = P(X \in A)\cdot P(X \in B) \qquad \forall A, B \in {\cal R}^{\mathbf{Z}}
\end{equation}
where $T$ is the shift operator on ${\mathbf{R}}^\mathbf{Z}.$

A well-known consequence of a strictly stationary sequence being mixing (in the ergodic-theoretic sense) is that it will also be ergodic.

In \cite{BP2009}, for an arbitrary fixed even integer $L$ satisfying $L\geq6$, the authors, Bradley and Pruss, constructed a strictly stationary, $(L-1)$-tuplewise independent, \emph{ergodic} sequence of real-valued random variables that fails to satisfy the Central Limit Theorem.  Working with the same construction as Bradley and Pruss, we shall show below that one can extend their result to construct a strictly stationary, $(L-1)$-tuplewise independent sequence that is mixing (in the ergodic-theoretic sense) and yet still fails to satisfy the Central Limit Theorem.

\subsection{Notations}\label{notations}

Before the main result is stated, a few notations are needed:

Convergence in distribution will be denoted by $\Rightarrow$.

For a given sequence $X:=(X_k,k\in\mathbf{Z})$ of random variables, for $n\in\mathbf{N}$, the partial sums will be denoted by

\begin{equation}
S_n := S(X,n) := X_1 + X_2 + \cdots + X_n
\end{equation}

If $a\leq b$ are integers and $U_a,U_{a+1},U_{a+2},\ldots,U_b$ are random variables, then the random vector $(U_a,U_{a+1},U_{a+2},\ldots,U_b)$ will be denoted as $U[a,b]$.  Similarly, if $u_a,u_{a+1},u_{a+2},\ldots,u_b$ are real numbers, then the real vector $(u_a,u_{a+1},u_{a+2},\ldots,u_b)$ will be denoted as $u[a,b]$.  Also, if $a$ is an integer and $U_a,U_{a+1},U_{a+2},\ldots$ are random variables (resp. $u_a,u_{a+1},u_{a+2},\ldots$ are real numbers), then the random sequence $(U_a,U_{a+1},U_{a+2},\ldots)$ (resp. $(u_a,u_{a+1},u_{a+2},\ldots)$) will be denoted as $U[a,\infty)$ (resp. $u[a,\infty)$).

For any random sequence $\eta:=(\eta_k, k\in\mathbf{Z})$ and any nonempty finite set $G\subset\mathbf{Z}$, the notation $\eta_G$ will refer to the random vector $(\eta_{g(1)},\eta_{g(2)},\ldots,\eta_{g(\ell)})$ where $\ell:=\text{card }G$ and $\eta_{g(1)},\eta_{g(2)},\ldots,\eta_{g(\ell)}$ are, in strictly increasing order, the elements of $G$.

\subsection{Main Result}

Here is the main result.

\begin{thm}\label{mainresult}
Suppose $M$ is an integer such that $M\geq2$. Then, there exists a strictly stationary sequence $\tilde{X}:= (\tilde{X}_k,k\in\mathbb{Z})$ with $E\tilde{X}_0=0$, and $E\tilde{X}_0^2=1/2$ satisfying the following conditions:
\begin{itemize}
\item[\textup{(A)}] $\tilde{X}$ satisfies M-tuplewise independence.
\item[\textup{(B)}] $\tilde{X}$ is mixing (in the ergodic-theoretic sense).
\item[\textup{(C)}] The random variables $|\tilde{X}_k|, k\in\mathbf{Z}$ are independent (and identically distributed).
\item[\textup{(D)}] For every infinite set, $Q\subset\mathbf{N}$, there exists an infinite set $T\subset Q$ and a nondegenerate, non-normal probability measure, $\tilde{\mu}$ on $(\mathbf{R}, {\cal R})$ such that $S(\tilde{X},n)/\sqrt{n}\Rightarrow\tilde{\mu}$, as $n\rightarrow\infty, n\in T$.
\end{itemize}
\end{thm}

Before delving into the proof, much background information (from \cite{BP2009} and \cite{RB}) will need to be introduced.

\section{Background Information}

Throughout the rest of this paper, once again (just as in \cite{BP2009}) let $L$ be an arbitrary fixed integer such that $L$ is {\bf even} and $L\geq6$. Then in order to prove Theorem \ref{mainresult}, it suffices to construct a strictly stationary, $(L-1)$-tuplewise independent random sequence that satisfies the properties described in the theorem.

This section will be divided into two main parts: relevant background information taken from \cite{RB} and background information from \cite{BP2009}

\subsection{}

We begin the background information with a few well known facts about strictly stationary sequences that are mixing (in the ergodic-theoretic sense).

\begin{prop}\label{mixing} Suppose $X:=(X_k, k \in \mathbf{Z})$ is a strictly stationary random sequence.  Then $X$ is mixing (in the ergodic-theoretic sense) if and only if the following holds:

For every choice of integers, $I_1 \leq I_2$ and $J_1 \leq J_2$, and every choice of sets $A_i \in {\cal R}$, $i = I_1, I_1+1, I_1+2..., I_2$ and $B_j \in {\cal R}$, $j = J_1, J_1+1, J_1+2..., J_2$ one has that
\begin{multline}
\lim_{n \to \infty} P(\{X_i\in A_i, I_1\leq i\leq I_2\}\cap\{X_{n+j}\in B_j, J_1\leq j\leq J_2\}) = \\
= P(X_i\in A_i, I_1\leq i\leq I_2)\cdot P(X_j\in B_j, J_1\leq j\leq J_2).
\end{multline}
\end{prop}
A proof of the above proposition can be found in \cite{RB}.
%\begin{mydef}Suppose $X:=(X_k,k\in\mathbf{Z})$ is a sequence of random variables on a probability space $(\Omega, {\cal F}, P)$
%\end{mydef}
\begin{cons}\label{randspace}\textup{Suppose $0<p<1$ and $W := (W_k,k\in\mathbf{Z})$ is a sequence of random variables.  Suppose $V := (V_k,k\in\mathbf{Z)}$ is a sequence of independent, identically distributed (i.i.d.) random variables with $P(V_0=1)=p$ and $P(V_0=0)=1-p$ that is independent of the sequence $W$.  Let $\{{\kappa}_j,j\in\mathbf{Z}\}$ be the (random) set of all integers $k$ such that $V_k = 1$, with}

\begin{equation}
...<{\kappa}_{-2}<{\kappa}_{-1}<{\kappa}_0\leq 0 < 1 \leq {\kappa}_1<{\kappa}_2<...
\end{equation}
\textup{Define the sequence $\tilde{W} := (\tilde{W}_k, k\in\mathbf{Z})$ as follows: For each $k\in\mathbf{Z}$,}
\begin{equation}
\tilde{W}_k = \left\{
\begin{array}{rl}
W_j & \text{if } k = \kappa_j \text{ for a given $j \in \mathbf{Z}$}\\
0 & \text{if } V_k = 0
\end{array} \right.
\end{equation}
\end{cons}

This sequence $\tilde{W}$ will be referred to as the ``sequence ${\cal S}(W,V)$".

\begin{lem}\label{26.3}In the context of the above Construction (with all assumptions there satisfied), if also the sequence W is strictly stationary and $k$-tuplewise independent (where $k$ is a positive integer), then the following statements hold:
\begin{itemize}
\item[\textup{(a)}] $\tilde{W}$ is strictly stationary.
\item[\textup{(b)}] $\tilde{W}$ is $k$-tuplewise independent.
\item[\textup{(c)}] For each $x \in \mathbf{R}, P(\tilde{W}_0 \leq x) = (1-p)\cdot I_{[0,\infty)}(x)+p\cdot P(W_0\leq x). $  If $E|W_0|<\infty$, then $E|\tilde{W}_0|<\infty$ and $E\tilde{W}_0=p \cdot EW_0$.  If $EW^2_0<\infty$, then $E\tilde{W}^2_0 = p \cdot EW^2_0$.
\item[\textup{(d)}] If $r$ is a positive integer and $E|W_0|^r<\infty$ then for each $n\geq 1$,
\[
E[S(\tilde{W},n)]^r = \sum_{j=1}^n\binom{n}{j}p^j(1-p)^{n-j}E[S(W,j)]^r.
\]
\end{itemize}
\end{lem}
Statement (b) holds by a trivial argument, while proofs that statements (a), (c) and (d) (as well as other similar information regarding the sequence $\tilde{W}$) can be found in Theorem 26.4 and its proof in \cite{RB}.

\subsection{}\label{backbp}

We now introduce some background information, taken from \cite{BP2009}.  For each positive integer $n$, let %$W^{(n)}:=(W^{(n)}_0,W^{(n)}_1,\ldots, W^{(n)}_{L^n-1},)$, 
$Y^{(n)}:=(Y^{(n)}_k,k\in\mathbf{Z}),\ \tau(n)$, and $X^{(n)}:=(X^{(n)}_k,k\in\mathbf{Z})$ be as in Steps 3.1, 3.5, and 3.6 of \cite{BP2009}.  That is:

The sequence $Y^{(n)}:=(Y^{(n)}_k, k\in\mathbf{Z})$ satisfies the following properties:

\begin{itemize}
\item[(i)] The $\mathbf{R}^{L^n}$-valued random vectors, $Y^{(n)}[jL^n,(j+1)L^n-1],\: j\in\mathbf{Z}$ are independent, and each of them has the same distribution (on $\mathbf{R}^{L^n}$).  For each $n$, let $\nu_n$ be this common distribution.
\item[(ii)] $Y^{(n)}$ satisfies $(L-1)$-tuplewise independence.
\item[(iii)] If $h$ and $j$ are integers such that $h\equiv j$ mod $L^n$, then the random sequences $Y^{(n)}[h,\infty)$ and $Y^{(n)}[j,\infty)$ have the same distribution.
\item[(iv)] Suppose $m$ is a positive integer such that $m>n$.  Then the $\mathbf{R}^{L^n}$-valued random vectors, $Y^{(m)}[jL^n,(j+1)L^n-1],\: j\in\mathbf{Z}$, (a), satisfy $(L-1)$-tuplewise independence (that is, every $L-1$ of these random vectors are independent, so that since $L$ is even and $L\geq6$, one has that  in particular, every \emph{five} of them are independent), and, (b), all have the same distribution $\nu_n$ (from (i)).
\item[(v)] Suppose $m$ is a positive integer such that $m>n$. If $h$ and $j$ are integers such that $h\equiv j$ mod $L^n$, then the random vectors, $Y^{(m)}[h,h+L^n-1]$ and $Y^{(m)}[j,j+L^n-1]$ have the same distribution.
\item[(vi)] Suppose $m$ is a positive integer such that $m>n$.  Then for any two integers $j$ and $\ell$ such that $|j-\ell|\geq2L^n$, the random vectors $Y^{(m)}[j,j+L^n-1]$ and $Y^{(m)}[\ell,\ell+L^n-1]$ are independent.
\end{itemize}

Let $\tau(n)$ be a random variable which takes its values in the set $\{0,1,\ldots,L^n-1\}$ and is uniformly distributed on that set (that is, $P(\tau(n)=j)=L^{-n}$ for each $j$ in that set), such that $\tau(n)$ is independent of the sequence $Y^{(n)}$.

Then let the random sequence $X^{(n)}:=(X^{(n)}_k, k\in\mathbf{Z})$ be defined as follows: For each $k\in\mathbf{Z}$,
\begin{equation*}
X^{(n)}_k:=Y^{(n)}_{k+\tau(n)}.
\end{equation*}

We repeat the following remark from \cite{BP2009} (Remark 3.7 there) here below:

\begin{rk}\label{rmk3.7}  Suppose $n\geq0$.  For any $j\in\mathbf{Z}$ and any set  $C\in {\cal R}^{\mathbf{N}}$, by the above information given on the sequences $Y^{(n)}, X^{(n)}$ and the random variable $\tau(n)$,
\begin{align*}
P(X^{(n)}[j,\infty)\in C)&=\sum_{i=0}^{L^n-1}P(X^{(n)}[j,\infty)\in C|\tau(n)=i)\cdot P(\tau(n)=i)\\
&=\sum_{i=0}^{L^n-1}P(Y^{(n)}[i+j,\infty)\in C|\tau(n)=i)\cdot 1/L^n\\
&=(1/L^n)\sum_{i=0}^{L^n-1}P(Y^{(n)}[i+j,\infty)\in C)\\
&=(1/L^n)\sum_{i=0}^{L^n-1}P(Y^{(n)}[i,\infty)\in C).
\end{align*}
\end{rk}

We now have the following Lemma below (Lemma 3.9 in \cite{BP2009}):

\begin{lem}\label{lemma3.9}  Suppose $n\geq0$.  Then, the following statements hold:
\begin{itemize}
\item[(i)] The random sequence $X^{(n)}$ is strictly stationary.
\item[(ii)] $X^{(n)}_0$ is uniformly distributed on the interval $[-\sqrt{3}, \sqrt{3}]$.
\item[(iii)] The sequence $X^{(n)}$ satisfies $(L-1)$-tuplewise independence.
\item[(iv)] The random variables $|X^{(n)}_k|,\: k\in\mathbf{Z}$ are independent.
\end{itemize}
\end{lem}

A proof to the above Lemma can of course be found in \cite{BP2009}.  Notice that by Remark \ref{rmk3.7} that one has that the sequence $X^{(n)}$ is strictly stationary.

We finish this section with one last remark taken again from \cite{BP2009} that will play a role in showing Property (D) in Theorem \ref{mainresult}.

\begin{rk}\label{step4.8}
Let $Z$ be a $N(0,1)$ random variable.  Then for any integer $j$ and any positive integer $h$, one has that
\begin{equation*}
E\left(Y^{(n)}_{j+1}+Y^{(n)}_{j+2}+\ldots+Y^{(n)}_{j+h}\right)^L\leq E[h^{1/2}Z]^L
\end{equation*}
for all $n\geq0.$

Furthermore, if $n$ is a nonnegative integer, $h$ is an integer such that $h\geq2L^{n+1}$, and $j$ is any integer, then for all $m\geq n+1$,
\begin{equation*}
E\left(Y^{(m)}_{j+1}+Y^{(m)}_{j+2}+\ldots+Y^{(m)}_{j+h}\right)^L\leq E[h^{1/2}Z]^L-L!\cdot2^{-L}\cdot(L^n)^{L/2}.
\end{equation*}

Thus, by a calculation similar to the one in Remark \ref{rmk3.7} above, for any positive integer $h$, one has that
\begin{equation*}
E[S(X^{(n)},h)]^L\leq E[h^{1/2}Z]^L
\end{equation*}
for all $n\geq0$.
Using the above equations and another calculation similar to the one in Remark \ref{rmk3.7}, one has that for any $n\geq0$, any $h\geq2L^{n+1}$, and any $m\geq n+1$,
\begin{equation}\label{Lthmoment}
E[S(X^{(m)},h)]^L\leq E[h^{1/2}Z]^L - L!\cdot2^{-L}\cdot(L^n)^{L/2}.
\end{equation}
\end{rk}

\section{A Construction}

We now construct a random sequence, $\tilde{X}$ that will satisfy Theorem \ref{mainresult}.  We begin by defining some random sequences.

For each nonnegative integer $n$, let $V^{(n)}:=\left(V^{(n)}_k, k\in\mathbf{Z}\right)$ be an i.i.d.\ sequence of random variables with $P(V^{(n)}_k = 1)=P(V^{(n)}_k = 0)=1/2$ that is independent of $\left(Y^{(n)}, \tau(n), X^{(n)}\right)$. 

Then, for each nonnegative integer $n$ (see Construction \ref{randspace}), define the sequence $\tilde{X}^{(n)}:=(\tilde{X}^{(n)}_k, k\in\mathbf{Z})$ to be the sequence ${\cal S}(X^{(n)},V^{(n)})$.

\begin{lem}\label{constr} Suppose $n\geq0$.  Then, the following statements hold:
\begin{itemize}
\item[(i)] The random sequence $\tilde{X}^{(n)}$ is strictly stationary.
\item[(ii)] For each $x\in\mathbf{R}$ one has that $P(\tilde{X}^{(n)}_0\leq x)=\frac{1}{2}[I_{[0, \infty)}(x)+P(X^{(n)}_0\leq x)]$ and $E\tilde{X}^{(n)}_0=0$, $E\left(\tilde{X}^{(n)}_0\right)^2=1/2$, and $E\left(\tilde{X}^{(n)}_0\right)^4=9/10$.
\item[(iii)] For any positive integer $h$, one has that
\[
E[S(\tilde{X}^{(n)},h)]^L = \sum_{j=1}^h\binom{h}{j}\left(\frac{1}{2}\right)^{h}E[S(X^{(n)},j)]^L
\]
\item[(iv)] The sequence $\tilde{X}^{(n)}$ satisfies $(L-1)$-tuplewise independence.
\item[(v)] The random variables $|\tilde{X}^{(n)}_k|,\: k\in\mathbf{Z}$ are independent.
\end{itemize}
\end{lem}

\begin{proof} Notice that statements (i), (ii), (iii), (iv) and (v) above follow directly from Lemma \ref{26.3}, Lemma \ref{lemma3.9} and elementary arguments.
\end{proof}

The random sequence $\tilde{X}$ will now be constructed, (in the same way as the sequence $X$ was constructed in \cite{BP2009}), with a standard argument.

Recall the notations in Section \ref{notations}.

For every nonempty, finite set $G \subset \mathbf{Z}$, the family of (finite dimensional) distributions of the random vectors, $\tilde{X}_G, n\geq0$ is tight by Lemma \ref{constr} above and Lemma \ref{lemma3.9}, and hence every subsequence of those distributions has a further subsequence that converges weakly to some probability measure (on $\mathbf{R}^{(\text{card }G)}$).  Since there exist only countably many finite subsets of $\mathbf{Z}$, employing a standard Cantor diagonal procedure, one obtains a (henceforth fixed) infinite set $\mathbf{Q}\subset\mathbf{N}$ and a family of probability measures $\tilde{\mu}_G$, for nonempty finite sets $G\subset\mathbf{Z}$ (with $\tilde{\mu}_G$ being a probability measure on $\mathbf{R}^{(\text{card } G)}$ for each such $G$), such that for \emph{every} nonempy finite set $G\subset\mathbf{Z}$, $\tilde{X}_G^{(n)}\Rightarrow\tilde{\mu}_G$ as $n\rightarrow\infty, n\in\mathbf{Q}$.  By an elementary argument, that family of probability measures $\tilde{\mu}_G$ satisfies the Kolmogorov consistency condition.  Applying the Kolmogorov Existence (Consistency) Theorem, let $\tilde{X}:=(\tilde{X}_k, k\in\mathbf{Z})$ be a sequence of random variables such that for every nonempty finite set $G\subset\mathbf{Z}$,
\begin{equation}\label{kolconsistent}
\tilde{X}_G^{(n)}\Rightarrow\tilde{X}_G\qquad \text{as } n\rightarrow\infty, n\in\mathbf{Q}.
\end{equation}

\section{Properties of $\tilde{X}$}

We now verify that the sequence $\tilde{X}$ satisfies the properties that are found in Theorem \ref{mainresult}.

Suppose $h$ and $j$ are any integers and $m$ is a nonnegative integer and consider the sets $G(1):=\{h, h+1, h+2,\ldots, h+m\}$ and $G(2):=\{j, j+1, j+2,\ldots, j+m\}$.  By strict stationarity of $\tilde{X}$ (see Lemma \ref{constr}), one has that the random vectors $\tilde{X}^{(n)}_{G(1)}$ and $\tilde{X}^{(n)}_{G(2)}$ have the same distribution.  Thus, by \eqref{kolconsistent} one must have that the random vectors $\tilde{X}_{G(1)}$ and $\tilde{X}_{G(2)}$ must have the same distribution, and hence, the sequence $\tilde{X}$ is strictly stationary.  Moreover, if $G:=\{0\}$, then again by Lemma \ref{constr} and by \eqref{kolconsistent}, one has that $E\tilde{X}_0=0$, and $E\tilde{X}_0^2=1/2$.

Next, by Lemma \ref{constr}(i)(ii)(iv), and by \eqref{kolconsistent} (where one considers finite subsets $G\subset \mathbf{Z}$ with card $G=L-1$), one has that the sequence $\tilde{X}$ must satisfy $L-1$-tuplewise independence.  Thus, $\tilde{X}$ satisfies property (A) in Theorem \ref{mainresult} (where $M=L-1$).  The proof of property (C) is similar to that of property (A), involving Lemma \ref{constr}(i)(ii)(v), \eqref{kolconsistent} and finite subsets $G\subset \mathbf{Z}$ of arbitrarily high cardinality.

The proofs of properties (B) and (D) will require more steps.

\section{Proof that $\tilde{X}$ is mixing (in the ergodic-theoretic sense)}

We will now show that the sequence $\tilde{X}$ satisfies property (B) in Theorem \ref{mainresult}; that is, we will now show that $\tilde{X}$ is mixing (in the ergodic-theoretic sense).

By Proposition \ref{mixing}, it suffices to show that for every choice of integers, $I_1 \leq I_2$ and $J_1 \leq J_2$, and every choice of sets $A_i \in \cal {R}$, $i = I_1, I_1+1, I_1+2,..., I_2$ and $B_j \in {\cal R}$, $j = J_1, J_1+1, J_1+2,..., J_2$ one has that
\begin{multline*}
\lim_{N \to \infty} P(\{\tilde{X}_i\in A_i, I_1\leq i\leq I_2\}\cap\{\tilde{X}_{N+j}\in B_j, J_1\leq j\leq J_2\})=\\
= P(\tilde{X}_i\in A_i, I_1\leq i\leq I_2)\cdot P(\tilde{X}_j\in B_j, J_1\leq j\leq J_2)
\end{multline*}

Notice that for each choice of integers, $I_1 < I_2$ and $J_1 < J_2$, there exists a nonnegative integer, $n$ such that $\text{max}\{I_2-I_1,J_2-J_1\}\leq L^n$.  Thus, by strict stationarity and a trivial argument (involving sets of probability 1), to show that the above equation holds (and hence show that $\tilde{X}$ is mixing (in the ergodic-theoretic sense)), it suffices to show that for any nonnegative integer $n$, and any choice of sets, $A_i \in {\cal R}, i\in \{0,1,..., L^n-1\}$, $B_j \in {\cal R}, j\in \{0,1,..., L^n-1\}$
one has that
\begin{multline}\label{goal}
\lim_{N \to \infty} P(\{\tilde{X}_i\in A_i, -L^n+1\leq i\leq 0\}\cap\{\tilde{X}_{N+j}\in B_j, 0\leq j\leq L^n-1\})=\\
=P(\tilde{X}_i\in A_i, -L^n+1\leq i\leq 0)\cdot P(\tilde{X}_j\in B_j, 0\leq j\leq L^n-1).
\end{multline}

Our goal, therefore, is to show that the above equation, \eqref{goal}, holds.  This will unfold over several steps.

\subsection{More Notation}\label{morenotation}

Throughout the rest of this section, much notation and conventions will be used.  We will assume the following:

Suppose $n$ is an arbitrary but fixed nonnegative integer, and suppose $N$ is a positive integer such that $N>2L^n$.

Let $A_i, i\in \{0,1,..., L^n-1\}, B_j, j\in \{0,1,..., L^n-1\}$ be sets in ${\cal R}$ and define the sets $A:=A_0\times A_1\times \cdots \times A_{L^n-1}$ and $B:=B_0\times B_1\times \cdots \times B_{L^n-1}$.

Let $v:=(v_0, v_1, ..., v_{L^n-1})$ and $w:=(w_0, w_1, ..., w_{L^n-1})$ be elements of $\{0, 1\}^{L^n}$.  

Suppose that there are $Q$ indices $i\in\{0,1,...L^n-1\}$ such that $v_i=1$.  Let $q(Q-1)<q(Q-2)<...<q(0)$ be these $Q$ indices and define the sets $S_v:=\{q(0),q(1), ..., q(Q-1)\}$ and $A_v:=A_{q(Q-1)}\times A_{q(Q-2)}\times \cdots \times A_{q(0)}$.  Similarly, suppose that there are $T$ indices $j\in\{0,1,...L^n-1\}$ such that $w_j=1$.  Let $t(1)<t(2)<...<t(T)$ be these $T$ indices and define the sets $S_w:=\{t(1),t(2), ..., t(T)\}$ and $B_w:=B_{t(1)}\times B_{t(2)}\times \cdots \times B_{t(T)}$.

\subsection{Strategy}\label{strategy}

Suppose $n$ is an arbitrary but fixed nonnegative integer.

There are two key tasks that we would like to complete in this section.  The first task is to show that for any positive integer $m$ satisfying $m>n\geq0$, the following three equations hold:
\begin{multline}\label{task1a}
\Large{P\Big(\substack{(V^{(m)}[-L^n+1,0]=v,\ \tilde{X}^{(m)}[-L^n+1,0]\in A)\text{ and } \\ (V^{(m)}[N,N+L^n-1]=w,\ \tilde{X}^{(m)}[N,N+L^n-1]\in B)}\Big)}\\=\Large{P\Big(\substack{(V^{(n)}[-L^n+1,0]=v,\ \tilde{X}^{(n)}[-L^n+1,0]\in A)\text{ and} \\ (V^{(n)}[N,N+L^n-1]=w,\ \tilde{X}^{(n)}[N,N+L^n-1]\in B)}\Big)},
\end{multline}
\begin{multline}\label{task1b}
P(V^{(m)}[-L^n+1,0]=v,\ \tilde{X}^{(m)}[-L^n+1,0]\in A)=\\=P(V^{(n)}[-L^n+1,0]=v,\ \tilde{X}^{(n)}[-L^n+1,0]\in A),
\end{multline}
\begin{multline}\label{task1c}
P(V^{(m)}[0,L^n-1]=w,\ \tilde{X}^{(m)}[0,L^n-1]\in B)=\\=P(V^{(n)}[0,L^n-1]=w,\ \tilde{X}^{(n)}[0,L^n-1]\in B)
\end{multline}

The second task is to show that as $N\rightarrow\infty$,
\begin{multline}\label{task2}
P((V^{(n)}[-L^n+1,0]=v,\ \tilde{X}^{(n)}[-L^n+1,0]\in A)\text{ and }(V^{(n)}[N,N+L^n-1]=w,\ \tilde{X}^{(n)}[N,N+L^n-1]\in B))\\ \longrightarrow P(V^{(n)}[-L^n+1,0]=v,\ \tilde{X}^{(n)}[-L^n+1,0]\in A)\cdot P(V^{(n)}[0,L^n-1]=w,\ \tilde{X}^{(n)}[0,L^n-1]\in B)
\end{multline}

Let us now recall our goal.  In order to show that the sequence $\tilde{X}$ is mixing (in the ergodic-theoretic sense), we must show that(see equation \eqref{goal} and the notations in Section \ref{morenotation} above),
\begin{equation}\label{goal*}
\lim_{N \to \infty} P(\tilde{X}[-L^n+1,0]\in A, \tilde{X}[N,N+L^n-1]\in B) = P(\tilde{X}[-L^n+1,0]\in A)\cdot P(\tilde{X}[0,L^n-1]\in B),
\end{equation}
where, once again, $n$ is some arbitrary but \emph{fixed} nonnegative integer.

Notice that if \eqref{task1a} is true, then, for all $m\geq n+1$, one has that
\begin{equation*}
P(\tilde{X}^{(m)}[-L^n+1,0]\in A, \tilde{X}^{(m)}[N,N+L^n-1]\in B)= P(\tilde{X}^{(n)}[-L^n+1,0]\in A,\tilde{X}^{(n)}[N,N+L^n-1]\in B)
\end{equation*}

Furthermore, if \eqref{task1b} and \eqref{task1c} are true then, for all $m\geq n+1$, one has that
\begin{equation*}
P(\tilde{X}^{(m)}[-L^n+1,0]\in A)\cdot P(\tilde{X}^{(m)}[0,L^n-1]\in B)=P(\tilde{X}^{(n)}[-L^n+1,0]\in A)\cdot P(\tilde{X}^{(n)}[0,L^n-1]\in B).
\end{equation*}

Thus, by an elementary argument and by \eqref{kolconsistent}, one would then have that:
\begin{align*}
P(\tilde{X}[-L^n+1,0]\in A, \tilde{X}[N,N+L^n-1]\in B) &= P(\tilde{X}^{(n)}[-L^n+1,0]\in A,\tilde{X}^{(n)}[N,N+L^n-1]\in B),\text{ and}\\
P(\tilde{X}[-L^n+1,0]\in A)\cdot P(\tilde{X}[0,L^n-1]\in B) &= P(\tilde{X}^{(n)}[-L^n+1,0]\in A)\cdot P(\tilde{X}^{(n)}[0,L^n-1]\in B).
\end{align*}

Hence, if \eqref{task2} is also true, then one has that \eqref{goal*} holds, and hence the sequence $\tilde{X}$ is mixing (in the ergodic-theoretic sense).

\subsection{Task 1}\label{Task 1}

The goal of this section will be to show that \eqref{task1a}, \eqref{task1b}, and \eqref{task1c} hold.

Suppose $m$ is a positive integer such that $0\leq n<m$, %$c:=(c_0, c_1, ..., c_{L^n-1})\in\{0, 1\}^{L^n-1}$ and suppose that $C_k, k\in \{0,1,..., L^n-1\}$ are sets in \mathcal{R}.  Then note that
%\begin{equation*}
%
%\end{equation*}

Suppose for now that for all $i\in\{0,1,...L^n-1\}-S_v$ and all $j\in\{0,1,...L^n-1\}-S_w$, one has that $0\in A_i$ and $0\in B_j$.  If this is not the case then note that \eqref{task1a}, \eqref{task1b}, \eqref{task1c} and \eqref{task2} hold trivially.  Under these assumptions, one has that
\begin{align*}
RHS[\eqref{task1a}]&=\sum_{k=1}^N\Large{P\left(\substack{(V^{(n)}[-L^n+1,0]=v,\ \tilde{X}^{(n)}[-L^n+1,0]\in A),\ V^{(n)}_1+\ldots+V^{(n)}_{N-1}=k-1\\ \text{ and }(V^{(n)}[N,N+L^n-1]=w,\ \tilde{X}^{(n)}[N,N+L^n-1]\in B)}\right)}\\
&=\sum_{k=1}^N\Large{P\left(\substack{(V^{(n)}[-L^n+1,0]=v,\ X^{(n)}[-Q+1,0]\in A_v),\ V^{(n)}_1+\ldots+V^{(n)}_{N-1}=k-1\\ \text{ and }(V^{(n)}[N,N+L^n-1]=w,\ X^{(n)}[k,k+T-1]\in B_w)}\right)}\\
&=\sum_{k=1}^N\left(\frac{1}{2}\right)^{2L^n}\binom{N-1}{k-1}\left(\frac{1}{2}\right)^{N-1}\cdot P(X^{(n)}[-Q+1,0]\in A_v, X^{(n)}[k,k+T-1]\in B_w)\\
&=\left(\frac{1}{2}\right)^{2L^n}\sum_{k=1}^N\binom{N-1}{k-1}\left(\frac{1}{2}\right)^{N-1}P(X^{(n)}[-Q+1,0]\in A_v, X^{(n)}[k,k+T-1]\in B_w).
\end{align*}

Notice that by a calculation similar to the one above, one also has that
\begin{equation*}
LHS[\eqref{task1a}]=\left(\frac{1}{2}\right)^{2L^n}\sum_{k=1}^N\binom{N-1}{k-1}\left(\frac{1}{2}\right)^{N-1}P(X^{(m)}[-Q+1,0]\in A_v, X^{(m)}[0,T-1]\in B_w),
\end{equation*}
so that to show \eqref{task1a}, it suffices to show that for each $k\in\{1,2,\ldots,N\}$,
\begin{multline}\label{minitask}
P(X^{(m)}[-Q+1,0]\in A_v, X^{(m)}[k,k+T-1]\in B_w)=\\=P(X^{(n)}[-Q+1,0]\in A_v, X^{(n)}[k,k+T-1]\in B_w)
\end{multline}
holds.

Let $A'_0:=\mathbf{R}, A'_1:=\mathbf{R},\ldots, A'_{L^n-Q-1}:=\mathbf{R},A'_{L^n-Q}:=A_{q(Q-1)},A'_{L^n-Q+1}:=A_{q(Q-2)},\ldots,A'_{L^n-1}:=A_{q(0)}$ and $B'_0:=B_{t(1)}, B'_1:=B_{t(2)},\ldots,B'_{T-1}:=B_{t(T)},B'_T:=\mathbf{R},B'_T:=\mathbf{R},B'_{T+1}:=\mathbf{R},\ldots,B'_{L^n-1}:=\mathbf{R}$.  Letting $A':=A'_0\times A'_1\times\cdots\times A'_{L^n-1}$ and $B':=B'_0\times B'_1\times\cdots\times B'_{L^n-1}$ notice that to show \eqref{minitask} and hence show \eqref{task1a}, it is enough to show that for each $k\in\{1,2,\ldots, N\}$
\begin{multline}\label{minitask1}
P(X^{(m)}[-L^n+1,0]\in A', X^{(m)}[k,k+L^n-1]\in B')=\\=P(X^{(n)}[-L^n+1,0]\in A', X^{(n)}[k,k+L^n-1]\in B')
\end{multline}
holds.

Suppose $k\in\{1,2,\ldots, N\}$.  Then,
\begin{align*}
RHS[\eqref{minitask1}]&=\frac{1}{L^n}\sum_{i=0}^{L^n-1}P(X^{(n)}[-L^n+1,0]\in A', X^{(n)}[k,k+L^n-1]\in B'|\tau(n)=i)\\
&=\frac{1}{L^n}\sum_{i=0}^{L^n-1}P(Y^{(n)}[i-L^n+1,i]\in A', Y^{(n)}[i+k,i+k+L^n-1]\in B').
\end{align*}

Notice by a calculation similar to the one above, we also have that
\begin{align*}
LHS[\eqref{minitask1}]&=\frac{1}{L^m}\sum_{i=0}^{L^m-1}P(X^{(m)}[-L^n+1,0]\in A', X^{(m)}[k,k+L^n-1]\in B'|\tau(m)=i)\\
&=\frac{1}{L^m}\sum_{i=0}^{L^m-1}P(Y^{(m)}[i-L^n+1,i]\in A', Y^{(m)}[i+k,i+k+L^n-1]\in B').
\end{align*}

Now, recall the background information given in Section \ref{backbp}.  Suppose $j$ is such that $i\equiv j$ mod $L^n$.  By Section \ref{backbp}(iv) and a simple argument, 
\begin{multline*}
P(Y^{(m)}[i-L^n+1,i]\in A', Y^{(m)}[i+k,i+k+L^n-1]\in B')=\\=P(Y^{(m)}[j-L^n+1,j]\in A', Y^{(m)}[j+k,j+k+L^n-1]\in B').
\end{multline*}
Hence,
\begin{equation*}
LHS[\eqref{minitask1}]=\frac{1}{L^n}\sum_{i=0}^{L^n-1}P(Y^{(m)}[i-L^n+1,i]\in A', Y^{(m)}[i+k,i+k+L^n-1]\in B').
\end{equation*}

Thus, in order to show \eqref{task1a}, it suffices to show that for each $k\in\{1,2,\ldots,N\}$ and each $i\in\{0,1,\ldots,L^n-1\}$, one has that
\begin{multline}\label{minitask2}
P(Y^{(m)}[i-L^n+1,i]\in A', Y^{(m)}[i+k,i+k+L^n-1]\in B')=\\=P(Y^{(n)}[i-L^n+1,i]\in A', Y^{(n)}[i+k,i+k+L^n-1]\in B').
\end{multline}

We have a couple of cases to consider, but before looking into them, it will be useful to introduce some new notation that will be used throughout the rest of the section.

For each $i \in \{0, 1,\ldots, L^n-2\}$, define the following sets: 
\begin{equation*}
A(i):=\overbrace{\mathbf{R}\times\mathbf{R}\times\cdots\times\mathbf{R}}^\text{$i+1$ times}\times A'_0\times A'_1\times\cdots A'_{L^n-2-i},
\end{equation*}
and 
\begin{equation*}
A'(i):=A'_{L^n-i-1}\times A'_{L^n-i}\times\cdots\times A'_{L^n-1}\times \underbrace{\mathbf{R}\times\mathbf{R}\times\cdots\times\mathbf{R}}_\text{$L^n-i-1$ times},
\end{equation*}
and let $A(L^n-1):=\mathbf{R}^{L^n}$ and $A'(L^n-1):=A'$.

Also, let $B(0):=B'$, $B'(0):=\mathbf{R}^{L^n}$, and for each $j \in \{1,2,\ldots, L^n-1\}$ define the following sets:
\begin{equation*}
B(j):=\overbrace{\mathbf{R}\times\mathbf{R}\times\cdots\times\mathbf{R}}^\text{j times}\times B'_0\times B'_1\times\cdots\times B'_{L^n-1-j},\: \text{and,}\\
\end{equation*}
\begin{equation*}
B'(j)\&:=B'_{L^n-j}\times B'_{L^n-j+1}\times\cdots\times B'_{L^n-1}\times\overbrace{\mathbf{R}\times\mathbf{R}\times\cdots\times\mathbf{R}}^\text{$L^n-j$ times}.
\end{equation*}

As a first case, let us suppose that $i, k$ are such that $1\leq i+k\leq L^n-1$.  Then by background information in section \ref{backbp},

\begin{align*}
RHS[\eqref{minitask2}]&=P(Y^{(n)}[i-L^n+1,i]\in A',\ Y^{(n)}[i+k,i+k+L^n-1]\in B')\\
&=\Large{P\left(\substack{Y^{(n)}[-L^n,-1]\in \mathbf{R}\times\mathbf{R}\times\cdots\times\mathbf{R}\times A'_0\times A'_1\times\cdots\times A'_{L^n-i-2},\\Y^{(n)}[0,L^n-1]\in A'_{L^n-i-1}\times A'_{L^n-i}\times\cdots\times A'_{L^n-1}\times\mathbf{R}\times\mathbf{R}\times\cdots\times\mathbf{R}\times B'_0\times B'_1\times\cdots\times B'_{L^n-i-k-1},\\ \text{and }Y^{(n)}[L^n,2L^n-1]\in B'_{L^n-i-k}\times B'_{L^n-i-k+1}\times\cdots\times B'_{L^n-1}\times\mathbf{R}\times\mathbf{R}\times\cdots\times\mathbf{R}} \right)}\\
&=\nu_n(A(i))\cdot \nu_n( A'_{L^n-i-1}\times A'_{L^n-i}\times\cdots\times A'_{L^n-1}\times\mathbf{R}\times\mathbf{R}\times\cdots\times\mathbf{R}\times B'_0\times B'_1\times\cdots\times B'_{L^n-i-k-1})\\ 
& \qquad \cdot  \nu_n( B'(i+k)).
\end{align*}

By a calculation similar to the one above, one also has that by background information on the sequence $Y^{(m)}$ found in Section \ref{backbp},
\begin{align*}
LHS[\eqref{minitask2}]=&=P(Y^{(m)}[i-L^n+1,i]\in A',\ Y^{(m)}[i+k,i+k+L^n-1]\in B')\\
&=\Large{P\left(\substack{Y^{(m)}[-L^n,-1]\in \mathbf{R}\times\mathbf{R}\times\cdots\mathbf{R}\times A'_0\times A'_1\times\cdots \times A'_{L^n-i-2},\\Y^{(m)}[0,L^n-1]\in A'_{L^n-i-1}\times A'_{L^n-i}\times\cdots\times A'_{L^n-1}\times\mathbf{R}\times\mathbf{R}\times\cdots\times\mathbf{R}\times B'_0\times B'_1\times\cdots\times B'_{L^n-i-k-1},\\ \text{and }Y^{(m)}[L^n,2L^n-1]\in B'_{L^n-i-k}\times B'_{L^n-i-k+1}\times\cdots\times B'_{L^n-1}\times\mathbf{R}\times\mathbf{R}\times\cdots\times\mathbf{R}} \right)}\\
&=\nu_n(A(i))\cdot \nu_n( A'_{L^n-i-1}\times A'_{L^n-i}\times\cdots\times A'_{L^n-1}\times\mathbf{R}\times\mathbf{R}\times\cdots\times\mathbf{R}\times B'_0\times B'_1\times\cdots\times B'_{L^n-i-k-1})\\
& \qquad \cdot \nu_n(B'(i+k)).
\end{align*}

That is, in the case that $i, k$ are such that $1\leq i+k\leq L^n-1$, one has that \eqref{minitask2} holds.

Now, as a second case, suppose that $i, k$ are such that $i+k\geq L^n$.  Then let $K$ be the positive integer such that $KL^n\leq i+k\leq (K+1)L^n-1$.  Let $K'$ be the integer such that $i+k=KL^n+K'$.  Then,
\begin{align*}
RHS[\eqref{minitask2}]&=\Large{P\left(\substack{Y^{(n)}[-L^n,-1]\in \mathbf{R}\times\mathbf{R}\times\cdots\mathbf{R}\times A'_0\times A'_1\times\cdots\times A'_{L^n-i-2},\\Y^{(n)}[0,L^n-1]\in A'_{L^n-i-1}\times A'_{L^n-i}\times\cdots\times A'_{L^n-1}\times\mathbf{R}\times\mathbf{R}\times\cdots\times\mathbf{R},\\Y^{(n)}[KL^n,(K+1)L^n-1]\in \mathbf{R}\times\mathbf{R}\times\cdots\times\mathbf{R}\times B'_0\times B'_1\times\cdots\times B'_{L^n-K'-1},\text{ and}\\ Y^{(n)}[(K+1)L^n,(K+2)L^n-1]\in B'_{L^n-K'}\times B'_{L^n-K'+1}\times\cdots\times B'_{L^n-1}\times\mathbf{R}\times\mathbf{R}\times\cdots\times\mathbf{R}}\right)}\\
&=\nu_n(A(i))\cdot \nu_n(A'(i)) \cdot \nu_n(B(K')) \cdot \nu_n(B'(K')).
\end{align*}

Again, by a calculation similar to the one above (see Section \ref{backbp}), one has that \eqref{minitask2} holds, and hence \eqref{task1a} holds.

To complete our task, we now need to show that \eqref{task1b} and \eqref{task1c} hold.  That is, we will show that
\begin{multline*}
P(V^{(m)}[-L^n+1,0]=v,\ \tilde{X}^{(m)}[-L^n+1,0]\in A)=\\=P(V^{(n)}[-L^n+1,0]=v,\ \tilde{X}^{(n)}[-L^n+1,0]\in A),
\end{multline*}
and
\begin{multline*}
P(V^{(m)}[0,L^n-1]=w,\ \tilde{X}^{(m)}[0,L^n-1]\in B)=\\=P(V^{(n)}[0,L^n-1]=w,\ \tilde{X}^{(n)}[0,L^n-1]\in B).
\end{multline*}

We begin by noticing that (see Section \ref{rmk3.7})
\begin{align}\label{goodcalc1}
P(V^{(n)}[-L^n+1,0]=v,\ \tilde{X}^{(n)}[-L^n+1,0]\in A)&=\left(\frac{1}{2}\right)^{L^n}P(X^{(n)}[-Q+1,0]\in A_v) \notag \\
&=\left(\frac{1}{2}\right)^{L^n}P(X^{(n)}[-L^n+1,0]\in A')\notag \\
&=\left(\frac{1}{2}\right)^{L^n}\frac{1}{L^n}\sum_{i=0}^{L^n-1}P(Y^{(n)}[i-L^n+1,i]\in A') \notag \\
&=\left(\frac{1}{2}\right)^{L^n}\frac{1}{L^n}\sum_{i=0}^{L^n-1}\nu_n(A(i))\cdot\nu_n(A'(i))
\end{align}
and that by background information in Section \ref{backbp} one also has that
\begin{align}\label{goodcalc2}
P(V^{(m)}[-L^n+1,0]=v, \tilde{X}^{(m)}[-L^n+1,0]\in A)&=\left(\frac{1}{2}\right)^{L^n}P(X^{(m)}[-Q+1,0]\in A_v) \notag \\
&=\left(\frac{1}{2}\right)^{L^n}P(X^{(m)}[-L^n+1,0]\in A')\notag \\
&=\left(\frac{1}{2}\right)^{L^n}\frac{1}{L^m}\sum_{i=0}^{L^m-1}P(Y^{(m)}[i-L^n+1,i]\in A') \notag \\
&=\left(\frac{1}{2}\right)^{L^n}\frac{1}{L^n}\sum_{i=0}^{L^n-1}P(Y^{(m)}[i-L^n+1,i]\in A') \notag \\
&=\left(\frac{1}{2}\right)^{L^n}\frac{1}{L^n}\sum_{i=0}^{L^n-1}\nu_n(A(i))\cdot\nu_n(A'(i))
\end{align}

and hence \eqref{task1b} holds.

By similar arguments and calculations as the ones made above in \eqref{goodcalc1} and \eqref{goodcalc2}, one has that
\begin{align}\label{goodcalc3}
P(V^{(n)}[0,L^n-1]=w,\ \tilde{X}^{(n)}[0,L^n-1]\in B)&=\left(\frac{1}{2}\right)^{L^n}P(X^{(n)}[0,T-1]\in B_w)\notag\\
&=\left(\frac{1}{2}\right)^{L^n}P(X^{(n)}[0,L^n-1]\in B')\notag\\
&=\left(\frac{1}{2}\right)^{L^n}\frac{1}{L^n}\sum_{j=0}^{L^n-1}\nu_n(B(j))\cdot\nu_n(B'(j)),
\end{align}
and
\begin{align*}
P(V^{(m)}[0,L^n-1]=w,\ \tilde{X}^{(m)}[0,L^n-1]\in B)&=\left(\frac{1}{2}\right)^{L^n}P(X^{(m)}[0,T-1]\in B_w)\\
&=\left(\frac{1}{2}\right)^{L^n}P(X^{(m)}[0,L^n-1]\in B')\\
&=\left(\frac{1}{2}\right)^{L^n}\frac{1}{L^n}\sum_{j=0}^{L^n-1}\nu_n(B(j))\cdot\nu_n(B'(j)),
\end{align*}
and hence \eqref{task1c} holds.

We now proceed to show that \eqref{task2} holds.

\subsection{Task 2}\label{Task2}

The goal of this section will be to show that \eqref{task2} holds, and hence (see Section \ref{strategy}) prove that the sequence $\tilde{X}$ is mixing (in the ergodic-theoretic sense).  That is, we must show that as $N\rightarrow\infty$
\begin{align*}
P((V^{(n)}[&-L^n+1,0]=v,\ \tilde{X}^{(n)}[-L^n+1,0]\in A)\text{ and} (V^{(n)}[N,N+L^n-1]=w,\ \tilde{X}^{(n)}[N,N+L^n-1]\in B))\notag\\ &\longrightarrow P(V^{(n)}[-L^n+1,0]=v,\ \tilde{X}^{(n)}[-L^n+1,0]\in A)\cdot P(V^{(n)}[0,L^n-1]=w,\ \tilde{X}^{(n)}[0,L^n-1]\in B)
\end{align*}

From calculations done in Task 1 above, one has that
\begin{align*}
P(&(V^{(n)}[-L^n+1,0]=v,\ \tilde{X}^{(n)}[-L^n+1,0]\in A)\text{ and} (V^{(n)}[N,N+L^n-1]=w,\ \tilde{X}^{(n)}[N,N+L^n-1]\in B))=\\
&=\left(\frac{1}{2}\right)^{2L^n}\sum_{k=1}^N \binom{N-1}{k-1}\left(\frac{1}{2}\right)^{N-1}P(X^{(n)}[-L^n+1,0]\in A'\text{ and }X^{(n)}[k,k+L^n-1]\in B')\\
&=\left(\frac{1}{2}\right)^{2L^n}\left(\frac{1}{L^n}\right)\sum_{k=1}^N \sum_{i=0}^{L^n-1}\binom{N-1}{k-1}\left(\frac{1}{2}\right)^{N-1}P(Y^{(n)}[i-L^n+1,i]\in A'\text{ and }Y^{(n)}[i+k,i+k+L^n-1]\in B'),
\end{align*}
and by \eqref{goodcalc1} and \eqref{goodcalc3}, one also has that
\begin{align*}
P(&V^{(n)}[-L^n+1,0]=v,\ \tilde{X}^{(n)}[-L^n+1,0]\in A)\cdot P(V^{(n)}[0,L^n-1]=w,\ \tilde{X}^{(n)}[0,L^n-1]\in B)=\\
&=\left(\frac{1}{2}\right)^{2L^n}\left(\frac{1}{L^n}\right)^2\cdot\sum_{i=0}^{L^n-1}\sum_{j=0}^{L^n-1}P(Y^{(n)}[i-L^n+1,i]\in A')\cdot P(Y^{(n)}[j,j+L^n-1]\in B')\\
&=\left(\frac{1}{2}\right)^{2L^n}\left(\frac{1}{L^n}\right)^2\cdot\sum_{i=0}^{L^n-1}\sum_{j=0}^{L^n-1}\nu_n(A(i))\cdot \nu_n(A'(i))\cdot \nu_n(B(j))\cdot \nu_n(B'(j)).
\end{align*}

Thus, in order to show \eqref{task2} and complete the proof that $\tilde{X}$ is mixing (in the ergodic-theoretic sense), it suffices to show that as $N\rightarrow\infty$,

\begin{align}\label{1star}
\sum_{k=1}^N & \sum_{i=0}^{L^n-1}\binom{N-1}{k-1}\left(\frac{1}{2}\right)^{N-1}P(Y^{(n)}[i-L^n+1,i]\in A'\text{ and }Y^{(n)}[i+k,i+k+L^n-1]\in B')\notag \\ \longrightarrow & \frac{1}{L^n}\sum_{i=0}^{L^n-1}\sum_{j=0}^{L^n-1}\nu_n( A(i))\cdot \nu_n( A'(i))\cdot \nu_n(B(i))\cdot \nu_n(B'(i))
\end{align}

Now, suppose $\beta_{N-1}$ is a binomial random variable with parameters $N-1$ and 1/2.  Then (see Section \ref{backbp}),

\begin{align}\label{2star}
&\sum_{k=1}^N \sum_{i=0}^{L^n-1}\binom{N-1}{k-1}\left(\frac{1}{2}\right)^{N-1}P(Y^{(n)}[i-L^n+1,i]\in A'\text{ and }Y^{(n)}[i+k,i+k+L^n-1]\in B') \notag\\
&=\sum_{k=0}^{N-1} \sum_{i=0}^{L^n-1}P(\beta_{N-1}=k)\cdot P(Y^{(n)}[i-L^n+1,i]\in A'\text{ and }Y^{(n)}[i+k+1,i+k+L^n]\in B')\notag \\
&=\sum_{k=0}^{L^n-1} \sum_{i=0}^{L^n-1}P(\beta_{N-1}=k)\cdot P(Y^{(n)}[i-L^n+1,i]\in A'\text{ and }Y^{(n)}[i+k+1,i+k+L^n]\in B')\notag\\
& \: +P(\beta_{N-1} \in \{L^n,2L^n,3L^n,\ldots\})\sum_{i=0}^{L^n-1}P(Y^{(n)}[i-L^n+1,i]\in A')\cdot P(Y^{(n)}[i+1,i+L^n]\in B')\notag \\
& \: +P(\beta_{N-1} \in \{L^n+1,2L^n+1,\ldots\})\sum_{i=0}^{L^n-1}P(Y^{(n)}[i-L^n+1,i]\in A')\cdot P(Y^{(n)}[i+2,i+L^n+1]\in B') \notag \\
& \: +\:\ldots \notag\\
& \: +P(\beta_{N-1} \in \{2L^n-2,3L^n-2,\ldots\})\sum_{i=0}^{L^n-1}P(Y^{(n)}[i-L^n+1,i]\in A')\cdot P(Y^{(n)}[i+L^n-1,i+2L^n-2]\in B')\notag \\
& \: +P(\beta_{N-1} \in \{2L^n-1,3L^n-1,\ldots\})\sum_{i=0}^{L^n-1}P(Y^{(n)}[i-L^n+1,i]\in A')\cdot P(Y^{(n)}[i,i+L^n-1]\in B').
\end{align}

We now make a couple observations.

First, notice that

\begin{align*}
&\sum_{k=0}^{L^n-1} \sum_{i=0}^{L^n-1}P(\beta_{N-1}=k)\cdot P(Y^{(n)}[i-L^n+1,i]\in A'\text{ and }Y^{(n)}[i+k,i+k+L^n-1]\in B')\\
&\leq \sum_{k=0}^{L^n-1} \sum_{i=0}^{L^n-1}P(\beta_{N-1}=k)\\
&=L^n\left(\frac{1}{2}\right)^{N-1}\sum_{k=0}^{L^n-1}\binom{N-1}{k}\\
&\longrightarrow 0 \qquad \text{as }N\rightarrow\infty.
\end{align*}

That is,

\begin{align}\label{claim1}
& \sum_{k=0}^{L^n-1} \sum_{i=0}^{L^n-1}P(\beta_{N-1}=k)\cdot P(Y^{(n)}[i-L^n+1,i]\in A'\text{ and }Y^{(n)}[i+k,i+k+L^n-1]\in B') \notag \\
&\longrightarrow 0 \qquad \text{as }N\rightarrow\infty.
\end{align}

Second, we know that since $\beta_{N-1}$ is a Binomial random variable with parameters $N-1$ and $1/2$, one has that as $N\rightarrow\infty$ ,
\begin{equation}\label{mod1}
P(\beta_{N-1}=k \text{ mod } L^n) \longrightarrow \frac{1}{L^n}
\end{equation}
for $k\in\{0,1,\ldots, L^n-1\}.$

Furthermore, for each $k\in\{0,1,\ldots, L^n-1\}$,
\begin{align*}
P(\beta_{N-1}=k \text{ mod } L^n)&= P(\beta_{N-1}\in \{k, k+L^n, k+2L^n,\ldots \}) \\
&=P(\{\beta_{N-1}=k\}\cup \{\beta_{N-1}\in \{k+L^n, k+2L^n, k+3L^n,\ldots \}\}) \\
&=P(\beta_{N-1}=k)+P(\beta_{N-1}\in \{k+L^n, k+2L^n, k+3L^n,\ldots \})\\
&=\binom{N-1}{k}\left(\frac{1}{2}\right)^{N-1}+P(\beta_{N-1}\in \{k+L^n, k+2L^n, k+3L^n,\ldots \})\\
&\leq N^k\left(\frac{1}{2}\right)^{N-1}+P(\beta_{N-1}\in \{k+L^n, k+2L^n, k+3L^n,\ldots \}).
\end{align*}

Thus, by \eqref{mod1} and since $N^k(1/2)^{N-1}\rightarrow0$, one has that for each $k\in\{0,1,\ldots, L^n-1\},$

\begin{equation}\label{claim2}
P(\beta_{N-1}\in\{k+L^n, k+2L^n, k+3L^n,\ldots,\}) \longrightarrow \frac{1}{L^n}
\end{equation}
as $N\rightarrow\infty.$

By \eqref{claim1} and \eqref{claim2}, as $N\rightarrow\infty$, one has that
\begin{align}\label{3star}
&\sum_{k=0}^{N-1} \sum_{i=0}^{L^n-1}P(\beta_{N-1}=k)\cdot P(Y^{(n)}[i-L^n+1,i]\in A'\text{ and }Y^{(n)}[i+k+1,i+k+L^n]\in B') \notag \\
&\longrightarrow\frac{1}{L^n}\sum_{i=0}^{L^n-1}\sum_{j=0}^{L^n-1}P(Y^{(n)}[i-L^n+1,i]\in A')\cdot P(Y^{(n)}[i+j,i+j+L^n-1]\in B').
\end{align}

Notice that for each $i\in\{0,1,\ldots,L^n-1\}$,(see Section \ref{backbp}) one has that
\begin{align*}
&\sum_{j=0}^{L^n-1}P(Y^{(n)}[i+j,i+j+L^n-1]\in B')\\
&=P(Y^{(n)}[i,i+L^n-1]\in B')+P(Y^{(n)}[i+1,i+1+L^n-1]\in B')+\ldots\\
&+P(Y^{(n)}[i+L^n-1,i+2L^n-2]\in B')\\
&=\sum_{j=0}^{L^n-1}\nu_n(B(j))\cdot\nu_n(B'(j)),
\end{align*}
and since
\begin{equation*}
\sum_{i=0}^{L^n-1}P(Y^{(n)}[i-L^n+1,i]\in A') = \sum_{i=0}^{L^n-1}\nu_n(A(i))\cdot\nu_n(A'(i)),
\end{equation*}
one has that
\begin{align*}
&\sum_{i=0}^{L^n-1}\sum_{j=0}^{L^n-1}P(Y^{(n)}[i-L^n+1,i]\in A')\cdot P(Y^{(n)}[i+j,i+j+L^n-1]\in B')\\
&=\sum_{i=0}^{L^n-1}\left[P(Y^{(n)}[i-L^n+1,i]\in A')\sum_{j=0}^{L^n-1}P(Y^{(n)}[i+j,i+j+L^n-1]\in B')\right]\\
&=\sum_{i=0}^{L^n-1}\left[P(Y^{(n)}[i-L^n+1,i]\in A')\sum_{j=0}^{L^n-1}\nu_n(B(j))\cdot\nu_n(B'(j))\right]\\
&=\sum_{i=0}^{L^n-1}\nu_n(A(i))\cdot\nu_n(A'(i))\sum_{j=0}^{L^n-1}\nu_n(B(j))\cdot\nu_n(B'(j))\\
&=\sum_{i=0}^{L^n-1}\sum_{j=0}^{L^n-1}\nu_n(A(i))\cdot\nu_n(A'(i))\cdot\nu_n(B(j))\cdot\nu_n(B'(j)).
\end{align*}

Thus, \eqref{3star} reads that as $N\rightarrow\infty$,

\begin{multline*}
\sum_{k=0}^{N-1} \sum_{i=0}^{L^n-1}P(\beta_{N-1}=k)\cdot P(Y^{(n)}[i-L^n+1,i]\in A'\text{ and }Y^{(n)}[i+k+1,i+k+L^n]\in B')\\
\longrightarrow\frac{1}{L^n}\left(\sum_{i=0}^{L^n-1}\sum_{j=0}^{L^n-1}\nu_n(A(i))\cdot\nu_n(A'(i))\cdot\nu_n(B(j))\cdot\nu_n(B'(j))\right).
\end{multline*}

Hence, via \eqref{2star}, one has that \eqref{1star} holds, and thus we have completed Task 2, thereby completing the proof that the sequence $\tilde{X}$ is mixing (in the ergodic-theoretic sense).

\section{Proof that the CLT fails to hold for $\tilde{X}$}

Recall the background information and notation found in Construction \ref{randspace}.  Throughout this section, suppose $Z$ is a $N(0,1)$ random variable and suppose that $\xi:=(\xi_k, k\in\mathbf{Z})$ is a sequence of independent, identically distributed (i.i.d.) $N(0,1)$ random variables.

We now present a few lemmas.

\begin{lem}\label{lemmaone}Suppose $n$ is any nonnegative integer.  Then, for any $h\geq4L^{n+1}$, one has that
\begin{equation}
E[S(\tilde{X},h)/\sqrt{h}]^L\leq h^{-L/2}\sum_{j=1}^h\binom{h}{j}\left(\frac{1}{2}\right)^h E(S(\xi,j))^L-h^{-L/2}\cdot L!\cdot2^{-L}\cdot(L^n)^{L/2}\sum_{j=2L^{n+1}}^h\binom{h}{j}\left(\frac{1}{2}\right)^h.
\end{equation}
\end{lem}\label{Lemma1}
\begin{proof}
By equation \eqref{Lthmoment}, and the fact that $\xi$ is a sequence i.i.d.\ $N(0,1)$ random variables, for any integer $n\geq0$, any $h\geq4L^{n+1}$, and any $m\geq n+1$, one has that
\begin{equation*}
E[S(\tilde{X}^{(m)},h)]^L\leq\sum_{j=1}^h\binom{h}{j}\left(\frac{1}{2}\right)^hE[S(\xi,j)]^L - L!\cdot2^{-L}\cdot(L^n)^{L/2}\sum_{j=2L^{n+1}}^h\binom{h}{j}\left(\frac{1}{2}\right)^h.
\end{equation*}
The desired result now follows by the above inequality and \eqref{kolconsistent}.
\end{proof}

\begin{lem}\label{lemmatwo}
For any $h\geq4L$, one has that $$E[S(\tilde{X},h)/\sqrt{h}]^L\leq h^{-L/2}\sum_{j=1}^h\binom{h}{j}\left(\frac{1}{2}\right)^h E(S(\xi,j))^L - 2^{-(5L+2)/2}\cdot L!\cdot L^{-L}.$$
\end{lem}
\begin{proof}
Suppose $h\geq4L$.  Let $n$ be the nonnegative integer such that $4L^{n+1}\leq h<4L^{n+2}$.

Since $h<4L^{n+2}$, one has that $(2h)^{-L/2}>(2\cdot 4L^{n+2})^{-L/2}=2^{-3L/2}\cdot (L^n)^{-L/2}\cdot L^{-L}$, so that, $(2h)^{-L/2}\cdot L!\cdot2^{-L}\cdot(L^n)^{L/2}\geq 2^{-5L/2}\cdot L!\cdot L^{-L}$.  Thus, since $h\geq4L^{n+1}$, by Lemma \ref{lemmaone}, it follows that
\begin{equation}\label{1}
E[S(\tilde{X},h)/\sqrt{h}]^L\leq h^{-L/2}\sum_{j=1}^h\binom{h}{j}\left(\frac{1}{2}\right)^h E(S(\xi,j))^L - 2^{-5L/2}\cdot L!\cdot L^{-L}\sum_{j=2L^{n+1}}^h\binom{h}{j}\left(\frac{1}{2}\right)^h.
\end{equation}

Now, if $\beta_h$ is a Binomial random variable with parameters $h$ and $\frac{1}{2}$, one has that $P(\beta_h=k)=P(\beta_h=n-k)$ for any $k\leq \lfloor h/2\rfloor$.  Then (again, since $h\geq4L^{n+1}$)note that
\begin{equation}\label{2}
\sum_{j=2L^{n+1}}^h\binom{h}{j}\left(\frac{1}{2}\right)^h=P(\beta_h\geq 2L^{n+1})\geq P(\beta_h\geq \Big\lfloor \frac{h}{2}\Big\rfloor)\geq \frac{1}{2}
\end{equation}

Thus, by \eqref{1} and \eqref{2}, the desired result follows.
\end{proof}

\begin{lem}\label{lemmathree} There exists a positive integer $H$ such that for all $h>H$, one has that $E[S(\tilde{X},h)/\sqrt{h}]^L\leq E(Z/\sqrt{2})^L-2^{-(5L+4)/2}\cdot L!\cdot L^{-L}.$
\end{lem}
\begin{proof}Suppose $V:=(V_k, k\in\mathbf{Z})$ be a sequence of i.i.d.\ random variables with $P(V_0=1)=P(V_0=0)=1/2$.  Then let $\tilde{\xi}:=(\tilde{\xi}_k, k\in\mathbf{Z})$ be the sequence ${\cal S}(\xi,V).$

Then, by Lemma \ref{26.3}, the Central Limit Theorem, and [\cite{B}, Theorem 25.7], one has that as $h\rightarrow\infty$,
\begin{equation}\label{xi1}
(S(\tilde{\xi},h)/\sqrt{h})^L \Rightarrow (Z/\sqrt{2})^L.
\end{equation}

Notice that since $\xi$ is a sequence of i.i.d. $N(0,1)$ random variables, for each postive integer $j$, the random variable $S(\xi,j)$ has the same distribution as the random variable $\sqrt{j}Z$ so that by Lemma \ref{26.3}, one has
\begin{align*}
E(S(\tilde{\xi},h)/\sqrt{h})^{L+2}&=h^{-(L+2)/2}\sum_{j=1}^h\binom{h}{j}\left(\frac{1}{2}\right)^hE(S(\xi,j))^{L+2}\\
&=h^{-(L+2)/2}\sum_{j=1}^h\binom{h}{j}\left(\frac{1}{2}\right)^hE(\sqrt{j}Z)^{L+2}\\
&\leq EZ^{L+2}\sum_{j=1}^h\binom{h}{j}\left(\frac{1}{2}\right)^h\\
&\leq EZ^{L+2} \\
&<\infty.
\end{align*}

Thus, $sup_hE(S(\tilde{\xi},h)/\sqrt{h})^{L+2}<\infty$ and hence, by \eqref{xi1}, one must have that as $n\rightarrow\infty$,
\begin{equation}\label{xi2}
E(S(\tilde{\xi},h)/\sqrt{h})^L \longrightarrow E(Z/\sqrt{2})^L.
\end{equation}

Hence, by \eqref{xi2} and Lemma \ref{26.3}(d), there exists a positive integer $h_0$ such that for all $h\geq h_0$,
\begin{equation}\label{xi3}
h^{-L/2}\sum_{j=1}^h\binom{h}{j}\left(\frac{1}{2}\right)^hE(S(\xi,j))^L < E(Z/\sqrt{2})^L + 2^{-(5L+4)/2}\cdot L!\cdot L^{-L}.
\end{equation}

Now, let $H:=$ max$\{4L, h_0\}.$  Then, by Lemma \ref{lemmatwo} and \eqref{xi3}, for all $h\geq H$, one has that
\begin{equation*}
E[S(\tilde{X},h)/\sqrt{h}]^L\leq E(Z/\sqrt{2})^L-2^{-(5L+4)/2}\cdot L!\cdot L^{-L},
\end{equation*}
which is the desired result.
\end{proof}

We are now ready to prove property (D) in Theorem 1.1.

Recall that $E\tilde{X}_0=0$, $E\tilde{X}_0^2=1/2$ and $E\tilde{X}_0^4=9/10$, and also by $(L-1)$-tuplewise independence (all shown in earlier sections), one has that 
\begin{equation}\label{mean}
\forall n\geq 1, \qquad E[S(\tilde{X}, n)/\sqrt{n}] = 0 \quad \textup{and}
\end{equation}
\begin{equation}\label{var}
\forall n\geq 1, \qquad E[S(\tilde{X}, n)/\sqrt{n}]^2 = 1/2.
\end{equation}

Further, by the standard argument in [B, p. 85, proof of Theorem 6.1], one has that for each $n\in \mathbf{N}$,
\begin{equation}
E[S(\tilde{X},n)]^4 = n\cdot E\tilde{X}_0^4+3n(n-1)(E\tilde{X}_0^2)^2 = n\cdot (9/10)+(3n^2-3n)\cdot (1/4)
\end{equation}
Hence
\begin{equation}\label{corr}
\forall n\geq 1, \qquad E\left(S(\tilde{X}, n))/\sqrt{n}\right)^4 = 3/4+(3/20n)<2.
\end{equation}
By \eqref{var} and Chebyshev's inequality, one has that the family of distributions of the normalized partial sums $S(\tilde{X},n)/\sqrt{n}, n\in\mathbf{N}$ is tight.

By tightness, for every infinite subset $Q$ of $\mathbf{N}$, there exists an infinite set $T\subset Q$ and a probability measure, $\tilde{\mu}$ on $\mathbf{R}$
(both $T$ and $\mu$ henceforth fixed) such that
\begin{equation}\label{conv}
S(\tilde{X},n)/\sqrt{n} \Rightarrow \tilde{\mu} \quad \text{as $n\rightarrow\infty, n\in T$}
\end{equation}
To complete the proof of property (D) in Theorem (1.1), our task is now to show that $\tilde{\mu}$ is neither degenerate nor normal.

Because of \eqref{corr} and \eqref{conv} and the Corollary in [B, p. 338], one has by \eqref{mean} and \eqref{var} that
\begin{equation}\label{above}
\int_{x\in\mathbf{R}}x\tilde{\mu}(\mathrm{d}x)=0\quad\text{and}\quad\int_{x\in\mathbf{R}}x^2\tilde{\mu}(\mathrm{d}x)=1/2.
\end{equation}
Hence the probabillity measure $\tilde{\mu}$ has positive variance and is therefore nondegenerate.

If $\tilde{\mu}$ were normal, then by the above equation, \eqref{above}, it would have to be the $N(0,1/2)$ distribution.  Then by \eqref{conv} and [\cite{B}, Theorem 25.7 and Theorem 25.11] that $\text{liminf}_{n\rightarrow\infty}E[S(\tilde{X},n)/\sqrt{n}]^L\geq E(Z/\sqrt{2})^L$ where $Z$ is a $N(0,1)$ random variable.  But this contradicts Lemma \ref{lemmathree}, and thus $\tilde{\mu}$ cannot be normal.  This completes the proof of property (D).

\section{Acknowledgements}

The author thanks Richard C. Bradley for not only suggesting the problem, but also for helpful advice and suggestions made along the way.


\begin{thebibliography}{99}
\bibitem[B]{B}P. Billingsley, \textit{Probability and Measure}, 3rd ed., Wiley, New York, 1995.
\bibitem[RB]{RB} R.C. Bradley, \textit{Introduction to Strong Mixing Conditions.} Volumes 1 and 3.  Kendrick Press, Herber City (Utah), 2007.
\bibitem[BP2009]{BP2009}R.C. Bradley \& A. R. Pruss, A strictly stationary, $M$-tuplewise independent counterexample to the Central Limit Theorem, \textit{Stochastic Processes and Their Application}, \textbf 119, 2009, pp. 3300-3318.
\end{thebibliography}
\end{document}